\documentclass[12pt,reqno]{amsart}

\usepackage{fixltx2e} % Get the latest LaTeX enhancements
\usepackage[utf8]{inputenc} % Customize to your encoding
\usepackage{amssymb, latexsym, stmaryrd, amsthm, dsfont, amsfonts, amsbsy, amsmath, mathrsfs} % Imprescindible
\usepackage{mathtools} % for some math typesetting tricks
\usepackage{bm, bbm} % for good bold math symbols
\usepackage{enumerate} % enhanced "enumerate" environments
\usepackage{verbatim} % for enhanced "comment" environment
\usepackage{url}   
\usepackage{lscape} % for \url command in bibliography
\usepackage{centernot}
\usepackage[dvipsnames]{xcolor}
\usepackage[colorinlistoftodos]{todonotes} % to add inline or margin comments
\usepackage{nicefrac} % For comparison
\usetikzlibrary{calc}

\usepackage{microtype,tikz,tikz-cd}  
\makeatletter % A correction needed for microtype:
\def\MT@register@subst@font{\MT@exp@one@n\MT@in@clist\font@name\MT@font@list
 \ifMT@inlist@\else\xdef\MT@font@list{\MT@font@list\font@name,}\fi}
\makeatother

\makeatletter % Custom \item to reference custom labels
\newcommand{\myitem}[1]{%
\item[(#1)]\protected@edef\@currentlabel{#1}%
}
\makeatother

\usepackage[pdftex,bookmarks,bookmarksnumbered,linktocpage,   %%  customize to your liking
         colorlinks,linkcolor=blue,citecolor=blue]{hyperref}
         
\usepackage[capitalize,noabbrev]{cleveref} %for better references. Use \cref instead of \ref without the need of specifying Theorem, Lemma, etc.

\newcommand{\bit}{\begin{itemize}}    % but see also \benbullet below
\newcommand{\eit}{\end{itemize}}
\newcommand{\ben}{\begin{enumerate}}
\newcommand{\een}{\end{enumerate}}
\newcommand{\benroman}{\ben[\normalfont (i)]}  % *
\let\eroman\een
\newcommand{\bde}{\begin{description}}
\newcommand{\ede}{\end{description}}
\let\oper=\mathbb                               %  operators

\newcommand{\III}{\oper{I}}                     %  class operators
\newcommand{\SSS}{\oper{S}}                     %  class operators
\newcommand{\resLK}{{\upharpoonright}_{\mathscr{L}_{\mathsf{K}}}}
\newcommand{\res}{{\upharpoonright}}

\theoremstyle{theorem}
\newtheorem{Theorem}{Theorem}[section]
\newtheorem{Theorem-n}{Theorem}

\newtheorem{Proposition}[Theorem]{Proposition}

\newtheorem{Modal Sahlqvist Theorem}[Theorem]{Modal Sahlqvist Theorem}
\newtheorem{Intuitionistic Sahlqvist Theorem}[Theorem]{Intuitionistic  Sahlqvist Theorem}
\newtheorem{Esakia Duality}[Theorem]{Esakia Duality}

\newtheorem{Main Lemma}[Theorem]{Main Lemma}
\newtheorem{Compactness Theorem}[Theorem]{Compactness Theorem}
\newtheorem{Los Theorem}[Theorem]{\LL o\'s' Theorem}
\newtheorem{Isbell Theorem}[Theorem]{Isbell's Zigzag Theorem}
\newtheorem{Diagram Lemma}[Theorem]{Diagram Lemma}

\newtheorem{Transfer Lemma}[Theorem]{Transfer Lemma}
\newtheorem{Subdirect Decomposition Theorem}[Theorem]{Subdirect Decomposition Theorem}

\newtheorem{Corollary}[Theorem]{Corollary}
\newtheorem{Claim}[Theorem]{Claim}

\theoremstyle{definition}
\newtheorem{Definition}[Theorem]{Definition}

\theoremstyle{remark}

\newtheorem{Remark}[Theorem]{Remark}

\crefname{Theorem}{Theorem}{Theorems}
\crefname{Proposition}{Proposition}{Propositions}
\crefname{Lemma}{Lemma}{Lemmas}
\crefname{Corollary}{Corollary}{Corollaries}
\crefname{Claim}{Claim}{Claims}
\crefname{Definition}{Definition}{Definitions}
\crefname{exa}{Example}{Examples}
\crefname{Remark}{Remark}{Remarks}
\crefname{Fact}{Fact}{Facts}
\crefname{exer}{Exercise}{Exercises}
\crefname{problem}{Problem}{Problems}

\AddToHook{env/Proposition/begin}{\crefalias{Theorem}{Proposition}}
\AddToHook{env/Lemma/begin}{\crefalias{Theorem}{Lemma}}
\AddToHook{env/Corollary/begin}{\crefalias{Theorem}{Corollary}}
\AddToHook{env/Claim/begin}{\crefalias{Theorem}{Claim}}
\AddToHook{env/Definition/begin}{\crefalias{Theorem}{Definition}}
\AddToHook{env/exa/begin}{\crefalias{Theorem}{exa}}
\AddToHook{env/Remark/begin}{\crefalias{Theorem}{Remark}}

\let\leq=\leqslant

\let\geq=\geqslant 

\usepackage[mathscr]{euscript}
 \let\mathscr\relax % just so we can load this and rsfs
\usepackage[scr]{rsfso}

\newcommand{\dom}{\mathsf{dom}}

\renewcommand{\int}{\mathsf{int}\,}

\bmdefine{\A}{A} 
\bmdefine{\C}{C}                                %  particular algebras
                              
\bmdefine{\2}{2}
\bmdefine{\B}{B}
\bmdefine{\D}{D}
\bmdefine{\E}{E}

\bmdefine{\Term}{T} 
\bmdefine{\Free}{F}
\bmdefine{\Fb}{F}
\usepackage{scalerel}

\newcommand{\K}{\mathsf{K}}
\newcommand{\M}{\mathsf{M}}

\newcommand{\PPP}{\mathbb{P}}

\newcommand{\PPU}{\mathbb{P}_{\!\textsc{\textup{u}}}^{}}

\newcommand{\ext}{\mathsf{ext}}
\newcommand{\extpp}{\mathsf{ext}_{\textsc{pp}}}
\newcommand{\extbang}{\mathsf{ext}_{\textsc{pp}!}}

\newcommand{\imp}{\mathsf{imp}}
\newcommand{\imppp}{\mathsf{imp}_{\textsc{pp}}}
\newcommand{\impbang}{\mathsf{imp}_{\textsc{pp}!}}

\let\LL\L %I redefined the original \L command as \LL (to be used for the initial letter in Los). So we can use \L for languages
\renewcommand{\L}{\mathscr{L}}

\newcommand{\F}{\mathcal{F}}

\DeclareUnicodeCharacter{001B}{}

\makeatletter
\renewenvironment{abstract}
  {%
    \small
    \begin{center}%
      {\bfseries \abstractname\par}%
    \end{center}%
  }
\makeatother

\begin{document}

  \title[A categorical description of Beth companions]{A categorical description of simple Beth companions}
  
  \author{Luca Carai, Miriam Kurtzhals, and Tommaso Moraschini}

\address{Luca Carai: Dipartimento di Matematica ``Federigo Enriques'', Universit\`a degli Studi di Milano, via Cesare Saldini 50, 20133 Milano, Italy}\email{luca.carai.uni@gmail.com}

\address{Miriam Kurtzhals and Tommaso Moraschini: Departament de Filosofia, Facultat de Filosofia, Universitat de Barcelona (UB), Carrer Montalegre, $6$, $08001$ Barcelona, Spain}
\email{mkurtzku7@alumnes.ub.edu and tommaso.moraschini@ub.edu}

  \maketitle

\begin{abstract}
A pp expansion of a quasivariety $\K$ is said to be \emph{simple} when it is of the form $\K[\L_\F]$. For instance, when $\K$ has the amalgamation property, all its pp expansions are simple. It is shown that the simple pp expansions of a quasivariety $\K$ coincide with the quasivarieties $\M$ for which the forgetful functor $U \colon \M \to \K$ is well defined and induces an isomorphism from $\M$ to a mono-reflective subcategory of $\K$. As a consequence, if a quasivariety $\K$ possesses a simple Beth companion $\M$, then $\M$ is the unique (up to term equivalence) quasivariety whose monomorphisms are regular that, moreover, satisfy the categorical description of simple pp expansions of $\K$ given above.
\end{abstract}

  \section{Expansions of quasivarieties}

We denote the class operators of closure under isomorphisms, subalgebras, direct products, finite direct products, and ultraproduts by $\III, \SSS, \PPP, \PPP^{<\omega}$, and $\PPU$, respectively. A class of similar algebras is 
said to be a \emph{quasivariety} when it is closed under $\III, \SSS, \PPP$, and $\PPU$ or, equivalently, when it can be axiomatized by a set of \emph{quasiequations}, that is, formulas of the form $\bigsqcap \Phi \to \varphi$, where $\Phi \cup \{ \varphi \}$ is a finite set of equations and $\bigsqcap$ the conjunction symbol (see, e.g.,  \cite[Thm.\ V.2.25]{BuSa00}). 

Let $\K$ be a quasivariety. A congruence $\theta$ of an algebra $\A$ is a $\K$-\emph{congruence} when $\A / \theta \in \K$. For every $X \subseteq A \times A$ there exists the least $\K$-congruence of $\A$ containing $X$, which we denote by $\mathsf{Cg}_\K^\A(X)$ (see, e.g., \cite[Sec.\ 1.4.4, p.\ 39]{Go98a}). Every quasivariety can be viewed as a category whose objects are its members  and whose arrows are the homomorphisms between them. 

\begin{Remark}\label{Rem : mono = one to one}
In every quasivariety, monomorphisms coincide with embeddings, i.e., injective homomorphisms (see, e.g., \cite[p.~222]{McK96}).
\qed
\end{Remark}
\subsection{The basic adjunction}    We denote the language of a quasivariety $\K$ by $\L_\K$. Given a quasivariety $\K$ and a language $\L \subseteq \L_\K$, we denote the $\L$-reduct of a member $\A$ of $\K$ by  $\A \res_\L$. Given a pair of  quasivarieties $\K$ and $\M$, the forgetful functor 
        $U \colon \mathsf{M} \to \mathsf{K}$ is well defined if and only if $\L_\K \subseteq \L_\M$ and  $\A \resLK \in \mathsf{K}$ for every $\A \in \mathsf{M}$. 

\begin{Definition}
    Let $\K$ and $\M$ be a pair of quasivarieties. We say that $\M$ is an \emph{expansion} of $\K$ when the forgetful functor $U \colon \M \to \K$ is well defined.
\end{Definition}

Let $\K$ be a quasivariety. We recall that for each nonempty set $X$ the free algebra $\boldsymbol{T}_\K(X)$ over $\K$ with set free generators $X$ belongs to $\K$ (see, e.g., \cite[Thm.~II.10.12]{BuSa00}). We will often identify a term $t(x_1, \dots, x_n)$ of $\K$ with $x_1, \dots, x_n \in X$ with its equivalence class in $\boldsymbol{T}_\K(X)$.

Consider an expansion $\M$ of a quasivariety $\K$. The forgetful functor $U \colon \M \to \K$ has a left adjoint $F$ which can be described as follows. With each $\A \in \K$ we associate a set of elements of pair of terms of $\K$ as follows:
\begin{align*}
    \mathsf{diag}^+(\A) &= \{ \langle t(a_1, \dots, a_n), s(a_1, \dots, a_n) \rangle : t(x_1, \dots, x_n) \text{ and }s(x_1, \dots, x_n) \text{ are terms of }\K,\\
            & \, \, \, \, \, \, \, \, \, \, \, \, a_1, \dots, a_n \in A, \text{ and }t^\A(a_1, \dots, a_n) = s^\A(a_1, \dots, a_n) \}.
\end{align*}
As $\M$ is an expansion of $\K$, we can view $\mathsf{diag}^+(\A)$ as a subset of $T_\M(A) \times T_\M(A)$. Then the algebra
\[
F(\A) = \boldsymbol{T}_\M(A) / \theta_\A, \text{ where }\theta_\A = \mathsf{Cg}_{\mathsf{M}}^{\boldsymbol{T}_{\mathsf{M}}(A) \color{black}}(\mathsf{diag}^+(\A))
\]
belongs to $\M$ because $\theta$ is an $\M$-congruence of $\boldsymbol{T}_\M(A)$ by definition. Furthermore, for each homomorphism $h \colon \A \to \B$ with $\A, \B \in \K$ let $F(h) \colon F(\A) \to F(\B)$ be the homomorphism defined for all $a_1, \dots, a_n \in A$ and terms $t(x_1, \dots, x_n)$ of $\M$ as $F(h)(t(a_1, \dots, a_n) / \theta_\A) = t(h(a_1), \dots, h(a_n)) / \theta_\B$. We call $F \colon \K \to \M$ the \emph{free extension functor} associated with $U \colon \M \to \K$.

We denote the unit and the counit of the adjunction $F \dashv U$ by $\eta \colon \mathsf{id}_\K \to UF$ and $\epsilon \colon FU \to \mathsf{id}_\M$, respectively. For every $\A \in \K$ the map $\eta_\A \colon \A \to UF(\A)$ is defined for every $a \in A$ as $\eta_\A(a) = a / \theta_\A$. Moreover, for every $\B \in \M$ the map $\epsilon_\B \colon FU(\B) \to \B$ is defined for all $b_1, \dots, b_n \in B$ and terms $t(x_1, \dots, x_n)$ of $\M$ as $\epsilon_\B(t(b_1, \dots, b_n) / \theta_{U(\B)}) = t^\B(b_1, \dots, b_n)$.

\subsection{Implicit operations}

An \emph{implicit operation} of a quasivariety $\K$ is a family of partial functions 
on the members of $\K$ that is globally preserved by homomorphisms and definable by a formula (see \cite[Sec.\ 3]{CKMIMPv2}).
More precisely, an $n$-ary \emph{operation} of $\mathsf{K}$ is a sequence $f = \langle f^\A : \A \in \mathsf{K} \rangle$, where each $f^\A \colon \dom(f^\A) \to A$ is a partial $n$-ary function on $A$ with domain $\mathsf{dom}(f^\A) \subseteq A^n$
that is  globally preserved by the homomorphisms between members of $\mathsf{K}$.
The latter means that for every homomorphism $h \colon \A \to \B$ with $\A, \B \in \K$ and $\langle a_1, \dots, a_n \rangle \in \mathsf{dom}(f^\A)$ we have
\[
\langle h(a_1), \dots, h(a_n) \rangle \in \mathsf{dom}(f^\B) \, \, \text{ and } \, \, h(f^\A(a_1, \dots, a_n)) = f^\B(h(a_1), \dots, h(a_n)).
\]
For $n \geq 1$, an
 % An 
 $n$-ary operation $f$  of $\K$ is said to be \emph{implicit} when it is defined by some first order formula $\varphi(x_1, \dots, x_n, y)$, in the sense that for all $\A \in \K$ and $a_1, \dots, a_n, b \in A$,
 \[
\A \vDash \varphi(a_1, \dots, a_n, b) \iff \langle a_1, \dots, a_n \rangle \in \mathsf{dom}(f^\A) \text{ and }f^\A(a_1, \dots, a_n) = b.
 \]
 For instance, “taking  inverses” is an implicit operation of the class of all monoids because it can be defined by the conjunction of equations $\varphi(x, y) = (xy \thickapprox 1) \sqcap (yx \thickapprox 1)$ and  monoid homomorphisms preserve  inverses when they exist.

Notably,  implicit operations admit a  description in terms of \emph{primitive positive formulas} (for short, \emph{pp formulas}), that is, formulas of the form $\exists x_1, \dots, x_n \varphi$, where $\varphi$ is a conjunction of equations. More precisely,  if $f$ is an implicit operation of a quasivariety $\K$, there exist implicit operations $f_1, \dots, f_n$ of $\K$ definable by pp formulas such that $f^\A = f_1^\A \cup \dots \cup f_n^\A$ for every $\A \in \K$ (see \cite[Cor.\ 3.10]{CKMIMPv2}).
Consequently, the pp definable implicit operations of $\K$ form the building blocks of all implicit operations of $\K$ and, therefore, we restrict our attention to them. The next result simplifies the task of determining whether a function can be defined by a pp formula  (see \cite[Thm.\ 6.3(4)]{CampVaggSemCon}).

\begin{Theorem}\label{Thm : campercholi vaggione}
  Let $\K$ be a class of algebras, $f \in \L_\K$, and $\L \subseteq \L_\K$. Then there exists a pp formula of $\L$ that defines $f$ in $\K$ if and only if $f$ is preserved by every homomorphism $h \colon \A\res_\L \to \B\res_\L$ with $\A, \B \in \PPU\PPP^{<\omega}(\K)$.
\end{Theorem}

In general, the implicit operations of a quasivariety $\K$ need not be componentwise total. Therefore, we say that an implicit operation $f$ of $\K$ is \emph{extendable} when for all $\A \in \K$ and  $\langle a_1, \dots, a_n \rangle \in \mathsf{dom}(f^\A)$ there exists an algebra $\B \in \K$ extending $\A$ such that $\langle a_1, \dots, a_n \rangle \in \mathsf{dom}(f^\B)$. The class of extendable implicit operations of $\K$ will be denoted by $\ext(\K)$, and that of pp definable extendable implicit operations of $\K$ by $\extpp(\K)$. The next result justifies the term  
``extendable'' (see   \cite[Prop.\ 8.1 and Thm.\ 8.4]{CKMIMPv2}).

\begin{Theorem}\label{Thm : extendable : universal : SI}
    Let $\K$ be a quasivariety and $\A \in \K$. Then there exists $\B \in \K$  with $\A \leq \B$ such that $f^\B$ is total and extends $f^\A$ for each $f  \in \ext(\K)$. 
\end{Theorem}

The task of constructing extendable implicit operations is simplified by the following observation (see   \cite[Cor.\ 3.11]{CKMIMPv2}).

\begin{Proposition}\label{Prop : extendability trick new}
Let $\K$ be a quasivariety and $\varphi(x_1, \dots, x_n, y)$ a pp formula. If every $\A$ can be extended to some $\B \in \K$ on which $\varphi$ defines a total $n$-ary function, then $\varphi$ defines an $n$-ary member of $\extpp(\K)$.  
\end{Proposition}

In order to add a family of implicit operations $\mathcal{F} \subseteq \extpp(\K)$ to a quasivariety $\K$, we proceed as follows. Let $\L_\F$ the language obtained by adding to $\L_\K$ a new $n$-ary function symbol $g_f$ for each $n$-ary $f \in \F$. Then we expand every member $\A$ of $\K$ in which $\{ f^\A : f \in \F \}$ is a family of total functions to an algebra $\A[\L_\F]$ in the language $\L_\F$ by interpreting  $g_f$ as $f^\A$ for each $f \in \F$. 
The \emph{pp expansion} of $\K$ induced by $\F$ is $\SSS(\K[\L_\F])$. We will make use of the following observation (see   \cite[Prop.\ 10.2]{CKMIMPv2}).

\begin{Proposition}\label{Prop : subreducts}
    Let $\M$ be a pp expansion of a quasivariety $\K$. Then $\A\resLK \in \K$ for every $\A \in \M$.
\end{Proposition}

\section{Simple Beth companions}

In this note, we shall focus on the following kind of pp expansions.

\begin{Definition}
A pp expansion of a quasivariety $\K$ is said to be \emph{simple} when it is of the form $\K[\L_\F]$ for some $\mathcal{F} \subseteq \extpp(\K)$. 
\end{Definition}

Simple pp expansions are relatively common, as witnessed by the following.

\begin{Theorem}
Every pp expansion of a quasivariety with the amalgamation property is simple.
\end{Theorem}

Another source of pp expansions 
derives from the following kind of implicit operations.

\begin{Definition}
Let $\K$ be a quasivariety and $f \in \imppp(\K)$. We say that $f$ has \emph{unique witnesses} when it can be defined by  pp formula 
\[
\exists z_1, \dots, z_m \varphi(z_1, \dots, z_m,x_1, \dots,x_n,y)
\]
such that for all $\A \in \K$ and $\langle a_1, \dots, a_n \rangle \in \dom(f^\A)$ there exists a unique tuple $\langle b_1, \dots, b_m\rangle \in A^m$ satisfying
\[
\A \vDash \varphi (b_1, \dots, b_m, a_1, \dots, a_n,f^{\A}(a_1,\dots, a_n)).
\]
We denote by $\impbang(\K)$ the subset of $\imppp(\K)$ consisting of the  definable implicit operations with unique witnesses. We also let $\extbang(\K)=\extpp(\K) \cap \impbang(\K)$.
\end{Definition}

We also recall that a pp expansion $\M$ of a quasivariety $\K$ is said to be a \emph{Beth companion} of $\K$ when monomorphisms are regular  in $\M$ (see \cite[Rmk. 6.4 \& Thm. 11.6]{CKMIMPv2}). Although a quasivariety $\K$ may lack a Beth companion (see, e.g., \cite[Thm.~14.17]{CKMIMPv2} and \cite[Thm.\ 6.1]{CKMMON}), when it possesses one, it must be essentially unique, as we proceed to illustrate.

Let $\M_1$ and $\M_2$ be a pair of pp expansions of a quasivariety $\K$.
For $i = 1,2$ let $T_i$ be the set of terms of $\M_i$ with variables in $\{x_n : n \in \mathbb{N}\}$. Let $\rho \colon \L_{\M_2} \to T_1$ be a map that preserves the arities. For each $\L_{\M_1}$-algebra $\A$ let $\rho(\A)$ be the $\L_{\M_2}$-algebra with universe $A$ such that $f^{\rho(\A)}=\rho(f)^\A$ for each function symbol $f$ in $\L_{\M_2}$. Similarly, given an arity-preserving map $\tau \colon \L_{\M_1} \to T_2$ and an $\L_{\M_2}$-algebra $\B$, we define an $\L_{\M_1}$-algebra $\tau(\B)$.
We say that $\M_1$ and $\M_2$ are \emph{faithfully term equivalent relative to $\K$} if there exist arity-preserving maps $\tau \colon \L_{\M_1} \to T_2$ and $\rho \colon \L_{\M_2} \to T_1$ such that
$\tau(f)=f(x_1, \dots, x_n)$ and $\rho(f)=f(x_1, \dots, x_n)$ for each $n$-ary function symbol $f$ in $\L_\K$, and
for all $\A \in \M_1$ and $\B \in \M_2$ we have
\benroman
\item\label{item:term eq 1} $\rho(\A) \in \M_2$;
\item\label{item:term eq 2} $\tau(\B) \in \M_1$;
\item\label{item:term eq 3} $\tau \rho (\A)=\A$;
\item\label{item:term eq 4} $\rho \tau(\B)=\B$.
\eroman

When they exist, Beth companions are essentially unique in the following sense.
\begin{Theorem}[\protect{\cite[Thm.\ 11.7]{CKMIMPv2}}]
All the Beth companions of a quasivariety $\K$ are faithfully term equivalent
relative to $\K$.
\end{Theorem}

We will make use of the following characterization of Beth companions (see \cite[Thm.~11.6]{CKMIMPv2}).

\begin{Theorem}\label{Thm : interpolation}
The following are equivalent for a pp expansion $\M$ of a quasivariety $\K$:
\benroman
\item $\M$ is a Beth companion of $\K$;
\item for all $\A \in \M$, $f \in \imp(\K)$, and $\langle a_1, \dots, a_n \rangle \in \mathsf{dom}(f^{\A\resLK})$ there exists a term $t$ of $\M$ such that 
\[
t^\A(a_1, \dots, a_n) = f^{\A\resLK}(a_1, \dots, a_n).
\]
\eroman
\end{Theorem}

As Beth companions are pp expansions, the concept of ``simplicity'' applies to them as well.

\begin{Definition}
A Beth companion of a quasivariety $\K$ is said to be \emph{simple} when it is a simple pp expansion of $\K$.
\end{Definition}

Every Beth companion obtained by adding extendable pp definable implicit operations with unique witnesses is simple, as we proceed to illustrate.

\begin{Theorem}
    Every Beth companion of a quasivariety $\K$ that is induced by a subset of $\extbang(\K)$ is simple.
\end{Theorem}

\begin{proof}
Let $\F \subseteq \extbang(\K)$ and consider the corresponding pp expansion $\SSS(\K[\L_\F])$. Assume that $\SSS(\K[\L_\F])$ is a Beth companion of $\K$. We need to prove that it is simple. To this end, it suffices to show that $\SSS(\K[\L_\F]) \subseteq \K[\L_\F]$, for in this case $\SSS(\K[\L_\F]) = \K[\L_\F]$. 

Consider $\A \in \SSS(\K[\L_\F])$.\ Then $\A\resLK \in \K$ by \cref{Prop : subreducts}. We will show that $f^{\A\resLK}$ is total and coincides with $g_f^\A$ for every $f \in \mathcal{F}$ because, in this case, we would get that $\A\resLK[\L_\F]$ is well defined and $\A = \A\resLK[\L_\F] \in \K[\L_\F]$, as desired.

To this end, consider an $n$-ary $f \in \F$. Since $\F \subseteq \extbang(\K)$ by assumption, the operation $f$ is defined by a pp formula 
\[
\psi = \exists z_1, \dots, z_m \varphi(z_1, \dots, z_m,x_1, \dots,x_n,y)
\]
of $\L_\K$ such that 
\begin{equation}\label{Eq : bang trick}
\begin{split}
     & \quad \text{for all }\C \in \K\text{ and }\langle c_1, \dots, c_n \rangle \in \dom(f^{\C})\text{ there exists a unique tuple}\\
&\langle d_1, \dots, d_m\rangle \in C^m\text{ satisfying }\C \vDash \varphi (d_1, \dots, d_m, c_1, \dots, c_n,f^{\C}(c_1,\dots, c_n)).
\end{split}
\end{equation}

for all $\C \in \K$ and $\langle c_1, \dots, c_n \rangle \in \dom(f^{\C})$ there exists a unique tuple $\langle d_1, \dots, d_m\rangle \in C^m$ satisfying
\[
\C \vDash \varphi (d_1, \dots, d_m, c_1, \dots, c_n,f^{\C}(c_1,\dots, c_n)).
\]

Consider $a_1, \dots, a_n \in A$. As $\A \in \SSS(\K[\L_\F])$, there exists $\B \in \K[\L_\F]$ with $\A \leq \B$. From $\B \in \K[\L_\F]$ it follows that $g_f^\B$ is defined by $\psi$. Together with $\A \leq \B$, this guarantees the existence of a  tuple $\langle b_1, \dots, b_m \rangle \in B^m$ such that
\begin{equation}\label{Eq : in B the formula phi holds}
    \B \vDash \varphi(b_1, \dots, b_m, a_1, \dots, a_n,g_f^{\A}(a_1,\dots, a_n)).
\end{equation}
To conclude the proof, it will be enough to show that  $b_1, \dots, b_m \in A$. 

For suppose this is the case. As $\varphi$ is a formula of $\L_\K$, $\A \leq \B$, and $a_1, \dots, a_n, g_f^{\A}(a_1, \dots, a_n) \in A$, the above display and the assumption that $b_1, \dots, b_m \in A$ imply 
\[
\A\resLK \vDash \varphi(b_1, \dots, b_m, a_1, \dots, a_n,g_f^{\A}(a_1,\dots, a_n)).
\]
Since $\psi$ defines $f$, we obtain $\langle a_1, \dots, a_n \rangle \in \mathsf{dom}(f^{\A\resLK})$ and $f^{\A\resLK}(a_1, \dots, a_n) = g_f^{\A}(a_1,\dots, a_n)$. Hence, we conclude that $f^{\A\resLK}$ is total and coincides with $g_f^{\A}$, as desired. 

Therefore, it only remains to show that $b_1, \dots, b_m \in A$. Consider a positive $k \leq m$. We will show that $b_k \in A$. Consider the pp formula
\[
\psi_k(x_1, \dots, x_n, y, z_k) = \exists z_1, \dots, z_{k-1}, z_{k+1}, \dots, z_m \varphi(z_1, \dots, z_m, x_1, \dots, x_n, y)
\]
of $\L_\K$. We will show that $\psi_k$ defines a member of $\imppp(\K)$. Since pp formulas are preserved by homomorphisms, it suffices to prove that $\psi_k$ defines a partial function on the members of $\K$. To this end, consider $\C \in \K$ and $c_1, \dots, c_{n+1}, d, e \in C$ such that
\[
\C \vDash \psi_k(c_1, \dots, c_{n+1}, d) \sqcap \psi_k(c_1, \dots, c_{n+1}, e).
\]
By the definition of $\psi_k$ there exist $p_1, \dots, p_{k-1}, p_{k+1}, \dots, p_m, q_1, \dots, q_{k-1}, q_{k+1}, \dots, q_m \in C$ such that
\begin{align*}
    \C &\vDash \varphi(p_1, \dots, p_{k-1}, d, p_{k+1}, \dots, p_m, c_1, \dots, c_n, c_{n+1});\\
    \C & \vDash \varphi(q_1, \dots, q_{k-1}, e, q_{k+1}, \dots, q_m, c_1, \dots, c_n, c_{n+1}).
\end{align*}
Together with the assumption that $\psi$ defines $f$, the above display yields $\langle c_1, \dots, c_n \rangle \in \mathsf{dom}(f^\C)$ and $c_{n+1} = f^\C(c_1, \dots, c_n)$. Consequently, \eqref{Eq : bang trick} ensures that $e = d$. Hence, we conclude that $\psi_k$ defines a member $f_k$ of $\imppp(\K)$. 

Recall that $\B$ belongs to the pp expansion $\SSS(\K[\L_\F])$ of $\K$. Then $\B\resLK \in \K$ by \cref{Prop : subreducts}. From \eqref{Eq : in B the formula phi holds} and the assumption that $\psi_k$ defines $f_k$ it follows that
\[
f_k^{\B\resLK}(a_1, \dots, a_n, g_f^{\A}(a_1, \dots, a_n)) = b_k.
\]
As $\SSS(\K[\L_\F])$ is a Beth companion of $\K$ by assumption, we can apply \cref{Thm : interpolation}, obtaining a term $t(x_1, \dots, x_{n+1})$ of $\SSS(\K[\L_\F])$ such that
\[
t^\B(a_1, \dots, a_n, g_f^{\A}(a_1, \dots, a_n)) = b_k.
\]
Since $\A \leq \B$ by assumption and $a_1, \dots, a_n, g_f^\A(a_1, \dots, a_n) \in A$, we conclude that $b_k \in A$.
\end{proof}

\section{The main result}

We recall that a full subcategory $\M$ of a category $\K$ is said \emph{reflective} when the inclusion functor $i \colon \M \to \K$ has a left adjoint. In this case, the counit of the resulting adjunction is a natural isomorphism
(see, e.g., \cite[Thm.~19.14(4)]{AHS06}). When, in addition, the unit is componentwise a monomorphism, we say that $\M$ is a \emph{mono-reflective} subcategory of $\K$ (see, e.g., \cite[Def.~16.1]{AHS06}). Lastly, given a functor $F \colon \K \to \M$, we denote the direct image of $\K$ under $F$, viewed as a subcategory of $\M$, by $F[\K]$. Clearly, $F[\K]$ is full if and only if so is $F$.

Our main result is the following categorical description of simple pp expansions.

\begin{Theorem}\label{Thm : main pp}
    Let $\M$ be an expansion of a quasivariety $\K$. Then the following  are equivalent:
    \benroman
\item\label{item : 1} $\M$ is a simple pp expansion of $\K$;
\item\label{item : 2} the unit of the adjunction $F \dashv U$ is componentwise a monomorphism and the counit is a natural isomorphism;
\item\label{item : 3} the forgetful functor $U \colon \M \to U[\M]$ is an isomorphism from $\M$ to a mono-reflective subcategory of $\K$.
    \eroman
\end{Theorem}

The following description of simple Beth companion is an immediate consequence of \cref{Thm : main pp}.

\begin{Corollary}
 Let $\M$ be an expansion  of a quasivariety $\K$. Then $\M$ is a simple Beth companion of $\K$ if and only if monomorphisms are regular in $\M$ and any of the equivalent conditions in \cref{Thm : main pp} holds.
\end{Corollary}

We shall now prove \cref{Thm : main pp}.

\begin{proof}
\eqref{item : 3}$\Rightarrow$\eqref{item : 2}: Straightforward.
 \eqref{item : 2}$\Rightarrow$\eqref{item : 3}: We begin by showing that the forgetful functor $U \colon \M \to U[\M]$ is an isomorphism. As it is always bijective on objects, it suffices to show that it is fully faithful. The latter holds because the counit of the adjunction $F \dashv U$ is 
 a natural isomorphism
 by assumption and this always guarantees that the full faithfulness of the right adjoint  (see, e.g., \cite[Thm.~19.14(4)]{AHS06}). 

Next, we prove that $U[\M]$ is a mono-reflective subcategory of $\K$. First, $U[\M]$ is a full subcategory of $\K$ because  $U \colon \M \to \K$ is full. As the unit of the adjunction $F \dashv U$ is componentwise a monomorphism and $U \colon \M \to U[\M]$ is an isomorphism, we deduce that $U[\M]$ is a mono-reflective subcategory of $\K$.

  \eqref{item : 2}$\Rightarrow$\eqref{item : 1}: We begin with the following observation.
  
\begin{Claim} \label{Claim : pp definability}
For every $n$-ary $f \in \L_\M$ there exists a pp formula $\varphi_f(x_1, \dots, x_n, y)$ of $\mathscr{L}_{\mathsf{K}}$ such that for all $\A \in \M$ and $a_1, \dots, a_n, b \in A$,
\[
f^{\A}(a_1, \dots, a_n) = b \Longleftrightarrow \A \vDash \varphi_f(a_1, \dots, a_n,b).
\]
\end{Claim}

\begin{proof}[Proof of the Claim]
As $\M$ is a quasivariety, it is closed under $\PPP$ and $\PPU$. Therefore, in view of \cref{Thm : campercholi vaggione}, 
    it suffices to show that $U \colon \M \to \K$ is full, which can be shown as in the proof of the implication  \eqref{item : 2}$\Rightarrow$\eqref{item : 3}.
\end{proof}

Next, we verify the following.

\begin{Claim} \label{Claim : extendability}
For each $f \in \L_\M$ the  formula $\varphi_f$ defines a member $f^*$ of $\extpp(\mathsf{K})$.
\end{Claim}

\begin{proof}[Proof of the Claim]
From \cref{Claim : pp definability} it follows that the pp formula $\varphi_f(x_1, \dots, x_n, y)$ defines a total $n$-ary function on each member of $\M$. As $\varphi_f$ is a formula of $\L_\K$, it also defines a total $n$-ary function on each member of $\{ U(\B) : \B \in \M \}$. Consequently, in view of \cref{Prop : extendability trick new}, it suffices to show that every member $\A$ of $\K$ embeds into an algebra of the form $U(\B)$ with $\B \in \M$. To this end, consider $\A \in \K$. By assumption the map $\epsilon_\A \colon \A \to UF(\A)$ is a monomorphism with $F(\A) \in \M$. By \cref{Rem : mono = one to one} we obtain that $\A$ embeds into $UF(\A)$. Thus, taking $\B = F(\A)$, we are done.
\end{proof}

By \cref{Claim : extendability} the set $\mathcal{F} = \{f^* : f \in \mathscr{L}_{\mathsf{M}} - \L_\K\}$
 induces a pp expansion $\SSS(\mathsf{K}[\mathscr{L}_{\mathcal{F}}])$ of $\K$. To conclude the proof, it only remains to show that this pp expansion is simple and coincides with $\M$. The next claim establishes both facts at once.
 
\begin{Claim}
We have $\mathsf{M} = \K[\mathscr{L}_{\mathcal{F}}] = \SSS(\mathsf{K}[\mathscr{L}_{\mathcal{F}}])$.
\end{Claim}

\begin{proof}[Proof of the Claim]
 As $\M$ is closed under $\SSS$ because it is a quasivariety, it suffices to show that $\M = \K[\L_\F]$. We begin with the inclusion from left to right. Consider $\A \in \M$. As $U(\A) \in \K$, it suffices to show that $U(\A)[\L_\F]$ is defined and coincides with $\A$ or, equivalently, that $f^{*U(\A)}$ is total and coincides with $f^\A$ for each $f \in \L_\M - \L_\K$. To this end, consider an $n$-ary $f \in \L_\M - \L_\K$ and $a_1, \dots, a_n \in A$. By \cref{Claim : pp definability} we have $\A \vDash \varphi_f(a_1, \dots, a_n, f^\A(a_1, \dots, a_n))$. Since $\varphi_f$ is a formula in $\L_\K$, this yields $U(\A) \vDash \varphi_f(a_1, \dots, a_n, f^\A(a_1, \dots, a_n))$. As $\varphi_f$ defines $f^*$ by \cref{Claim : extendability}, we conclude that $\langle a_1, \dots, a_n \rangle \in f^{*U(\A)}$ and $f^{*U(\A)}(a_1, \dots, a_n) = f^\A(a_1, \dots, a_n)$. Hence, $f^{*U(\A)}$ is total and coincides with $f^\A$.
 
 Next, we prove the inclusion $\K[\L_\F] \subseteq \M$.  Consider $\A\in \K[\L_\F]$. By the definition of $\K[\L_\F]$ we have $\A\res_{\L_\K} \in \K$ and $\A = \A\res_{\L_\K}[\L_\F]$. 
 It will be enough to prove that the map $h \colon \A \to F(\A\res_{\L_\K})$ defined for every $a \in A$ as $h(a) = \eta_{\A\res_{\L_\K}}(a)$ is an embedding. 
 For suppose this is the case. Then $\A \in \III\SSS(F(\A\res_{\L_\K}))$. Since $F(\A\res_{\L_\K}) \in \M$ and $\M$ is closed under $\III\SSS$ because it is a quasivariety, we conclude that $\A \in \M$, as desired.

Recall that $\eta_{\A\res_{\L_\K}}$ is by assumption a monomorphism (i.e., an embedding by \cref{Rem : mono = one to one}). The definition of $h$ guarantees that it is a well-defined injective map that preserves the operations in $\L_\K$. It only remains to prove that $h$ preserves the operations in $\L_\M - \L_\K$ as well. Consider an $n$-ary $f \in \L_\M - \L_\K$ and $a_1, \dots, a_n \in A$. Since $\A = \A\res_{\L_\K}[\L_\F]$ and $f^*$ is defined by $\varphi_f$ by \cref{Claim : extendability} we have
\[
\A\res_{\L_\K} \vDash \varphi_f(a_1, \dots, a_n, f^{*\A\res_{\L_\K}}(a_1, \dots, a_n)).
\]
As $\varphi_f$ is a pp formula in $\L_\K$ and $h$ preserves the operations of $\L_\K$, we obtain that $h$ preserves $\varphi$. Together with the above display, this implies
\[
F(\A\res_{\L_\K}) \vDash \varphi_f(h(a_1), \dots, h(a_n), h(f^{*\A\res_{\L_\K}}(a_1, \dots, a_n))).
\]
By \cref{Claim : pp definability} and $F(\A\res_{\L_\K}) \in \M$ we conclude that 
\[
h(f^{*\A\res_{\L_\K}}(a_1, \dots, a_n)) = f^{F(\A\res_{\L_\K})}(h(a_1), \dots, h(a_n)).
\]
Hence, $h \colon \A \to F(\A\res_{\L_\K})$ is an embedding, as desired.
\end{proof}

\eqref{item : 1}$\Rightarrow$\eqref{item : 2}:    Let $\mathsf{M}$ be a simple pp expansion of $\K$ of the form $\mathsf{K}[\mathscr{L}_{\mathcal{F}}]$ for some $\F \subseteq \extpp(\K)$. 
    We begin by showing that each component of the unit $\eta$ is a monomorphism. Consider $\A \in \mathsf{K}$. Since $\mathcal{F} \subseteq \mathsf{ext}_{\textsc{pp}}(\mathsf{K})$, there exists $\B \in \K[\L_\F]$ for which the inclusion map
    $i \colon \A \to U(\B)$ is a well-defined embedding. As $F \dashv U$, there exists a homomorphism $g \colon F(\A) \to \B$ such that $U(g) \circ \eta_\A = i$ (see, e.g., \cite[Prop.~19.7(2)]{AHS06}). Since $i$ is an embedding, so is $\eta_\A$. Consequently, $\eta_\A$ is a monomorphism by \cref{Rem : mono = one to one}. 
    
    Next, we prove that the counit $\epsilon$ is a natural isomorphism. To this end, it suffices to show that $U \colon \M \to \K$ is fully faithful (see, e.g., \cite[Thm.~19.14(4)]{AHS06}). As the forgetful functor is always faithful, it only remains to show that it is full. Consider $\A, \B \in \M = \K[\L_\F]$ and a homomorphism $h \colon U(\A) \to U(\B)$. From  \cite[Prop.\ 9.5]{CKMIMPv2} it follows that $h$ is also a homomorphism from $\A$ to $\B$. Hence, $U$ is full, as desired.
\end{proof}

We close this note by observing that the property of ``being simple'' is preserved by faithful term equivalences between pp expansions.

\begin{Corollary}
Let $\K$ be a quasivariety and $\M_1,\M_2$ a pair of pp expansions of $\K$ that are faithfully term equivalent relative to $\K$. The $\M_1$ is simple if and only if so is $\M_2$.
\end{Corollary}

\begin{proof}
Let $\tau \colon \L_{\M_1} \to T_2$ and $\rho \colon \L_{\M_2} \to T_1$ witness the faithful term equivalence, and
assume that $\M_1$ is a simple pp expansion. Then \cref{Thm : main pp} implies that the forgetful functor $U_1 \colon \M_1 \to \K$ restricts to an isomorphism from $\M_1$ to a mono-reflective subcategory $U_1[\M_1]$ of $\K$. Let $U_2 \colon \M_2 \to \K$ be the forgetful functor. Let also $\tau \colon \M_2 \to \M_1$ be the functor that maps $\A \in \M_2$ to $\tau(\A) \in \M_1$ and is the identity on morphisms. Similarly, define $\rho \colon \M_1 \to \M_2$ as the functor that maps $\A \in \M_1$ to $\rho(\A) \in \M_2$ and is the identity on morphisms. Since $\tau$ and $\rho$ witness a term equivalence relative to $\K$, we have that these two functors are isomorphisms of categories that are inverses of each other and satisfy $U_1 \circ \tau =U_2$. It follows that $U_1[\M_1]=U_2[\M_2]$, and so $U_2[\M_2]$ is a mono-reflective subcategory of $\K$. Moreover, as $\tau \colon \M_2 \to \M_1$ and $U_1 \colon \M_1 \to U_1[\M_1]$ are isomorphisms and $U_1 \circ \tau =U_2$ (the latter because the term equivalence witnessed by $\tau$ and $\rho$ is faithful), we obtain that $U_2 \colon \M_2 \to U_2[\M_2]$ is an isomorphism as well. Therefore, \cref{Thm : main pp} yields that $\M_2$ is a simple pp expansion of $\K$.
\end{proof}

\

\subsection*{Acknowledgments} We are grateful to Ivan Di Liberti and Luca Reggio for many interesting conversations on the topic of this note.


\begin{thebibliography}{CKM26}

\bibitem[AHS06]{AHS06}
J.~Ad{\'a}mek, H.~Herrlich, and G.~E. Strecker.
\newblock Abstract and concrete categories: the joy of cats.
\newblock {\em Repr. Theory Appl. Categ.}, (17):1--507, 2006.

\bibitem[BS12]{BuSa00}
S.~Burris and H.~P. Sankappanavar.
\newblock {\em A Course in Universal Algebra}.
\newblock 2012.
\newblock The millennium edition, available online.

\bibitem[CKM25]{CKMIMPv2}
L.~Carai, M.~Kurtzhals, and T.~Moraschini.
\newblock The theory of implicit operations.
\newblock Available at \url{https://arxiv.org/pdf/2512.14326v2}, 2025.

\bibitem[CKM26]{CKMMON}
L.~Carai, M.~Kurtzhals, and T.~Moraschini.
\newblock Implicit operations in varieties of commutative monoids.
\newblock Submitted, available at \url{https://arxiv.org/pdf/2603.13916}, 2026.

\bibitem[CV15]{CampVaggSemCon}
M.~A. Campercholi and D.~J. Vaggione.
\newblock Semantical conditions for the definability of functions and relations.
\newblock {\em Algebra Universalis}, 76:71--98, 2015.

\bibitem[Gor98]{Go98a}
V.~A. Gorbunov.
\newblock {\em Algebraic theory of quasivarieties}.
\newblock Siberian School of Algebra and Logic. Consultants Bureau, New York, 1998.
\newblock Translated from the Russian.

\bibitem[McK96]{McK96}
R.~McKenzie.
\newblock An algebraic version of categorical equivalence for varieties and more general algebraic categories.
\newblock In A.~Ursini and P.~Aglian\`{o}, editors, {\em Logic and Algebra}, volume 180 of {\em Lecture Notes in Pure and Applied Mathematics}, pages 211--243. Marcel Dekker, Inc., 1996.

\end{thebibliography}
\end{document}